\documentclass[12pt,draft]{amsart}
\usepackage[all]{xy}
\usepackage{amsfonts}
\usepackage{amssymb}
\usepackage{amsthm}

\usepackage[english]{babel}
\usepackage{enumerate}

\usepackage{color}

\newtheorem{teo}{Theorem}[section]
\newtheorem{lem}[teo]{Lemma}

\newtheorem{cor}[teo]{Corollary}

\newtheorem{dfn}[teo]{Definition}
\newtheorem{rmk}[teo]{Remark}

\newtheorem*{theoremA}{Theorem A}
\newtheorem*{theoremB}{Theorem B}

\def\<{\langle}
\def\>{\rangle}

\def\a{\alpha}
\def\b{\beta}

\def\t{\tau}

\def\f{{\varphi}}

\def\Z{{\mathbb Z}}

\def\End{\mathop{\rm End}\nolimits}
\def\Ker{\mathop{\rm Ker}\nolimits}
\def\Im{\mathop{\rm Im}\nolimits}

\def\Coker{\operatorname{Coker}}

\def\1{\mathbf 1}

 \def\rk{\operatorname{rk}}
\def\urk{\operatorname{urk}}

\def\St{\operatorname{St}}

\begin{document}
\title[Twisted conjugacy and finite rank]
{Twisted conjugacy in residually finite groups of finite Pr\"ufer rank}

\author{Evgenij Troitsky}
\thanks{The work is supported by the Russian Science Foundation under grant 21-11-00080.}
\address{Moscow Center for Fundamental and Applied Mathematics, MSU Department,\newline
	Dept. of Mech. and Math., Lomonosov Moscow State University, 119991 Moscow, Russia}
\email{troitsky@mech.math.msu.su}

\begin{abstract}
Suppose, $G$ is a residually finite group of finite upper rank admitting an automorphism $\varphi$ with finite Reidemeister number $R(\varphi)$ (the number of $\varphi$-twisted conjugacy classes). 
We prove that such $G$ is soluble-by-finite (in other words, any residually finite group of finite upper rank, which is not soluble-by-finite, has the $R_\infty$ property). 

This reduction is the first step in the proof of the second main theorem of the paper:
suppose, $G$ is a residually finite group of finite Pr\"ufer rank and $\varphi$ is its automorphism with 
$R(\varphi)<\infty$; then $R(\varphi)$ is equal to the number of equivalence classes of finite-dimensional irreducible unitary representations
of  $G$, which are fixed points of the dual map $\widehat{\varphi}:[\rho]\mapsto [\rho\circ \varphi]$
(i.e., we prove the TBFT$_f$, the finite version of the conjecture about the
 twisted Burnside-Frobenius theorem, for such groups).
\end{abstract}

\maketitle

\section*{Introduction}

If $d(H)$ denotes the minimal number of generators of a group $H$, then the \emph{(Pr\"ufer) rank} of a group $G$
is 
$$
\rk(G):=\sup\{d(H)\colon H \mbox{ is a finitely generated subgroup of } G\}
$$
and the \emph{upper rank} of $G$ is
$$
\urk(G):=\sup\{\rk(\overline{G})\colon \overline{G} \mbox{ is a finite quotient of } G\}.
$$
Subgroups of finite rank attracted much interest since the end of 80's,
in particular, in connection with the
subgroup growth and related questions \cite{LubotSegalBook}. 

The main results of the present paper are as follows.   

\begin{theoremA}
Suppose that $G$ is a residually finite group
of finite upper rank and $\f:G\to G$ is an automorphism with
$R(\f)<\infty$. 
Then $G$ is a soluble-by-finite group. 
\end{theoremA}

A reformulation is: 
any residually finite group of finite upper rank, which is not soluble-by-finite, has the $R_\infty$ \emph{property}.

\begin{theoremB}
Suppose, $G$ is a   residually finite group
of finite rank. Then TBFT$_f$ is true for $G$.
\end{theoremB}

Theorem A may be considered as a development of \cite{Jabara} and at the same time as a
supporting result for the following conjecture \cite{FelTroJGT}: a residually finite finitely
generated group with an automorphism with finite Reidemeister number is soluble-by-finite.  
The proof of Theorem A uses in particular a result of A.~Jaikin-Zapirain \cite{Jaikin2002}.

The initial conjecture about the TBFT was formulated by A.~Fel'shtyn and R.~Hill
as follows: if the Reidemeister number is finite then it coincides with the 
number of equivalence classes of irreducible unitary representations
of  $G$, which are fixed points of the dual map $[\rho]\mapsto [\rho\circ \varphi]$
(see \cite{FelshB}).
A counterexample to TBFT from \cite{FelTroVer} caused a modified version,
namely TBFT$_f$, where only  finite-dimensional representations are taken
into account.
The widest class of groups with TBFT$_f$, namely polycyclic-by-finite, was detected
in \cite{polyc}. 
Then we have found several counterexamples, in particular,
in \cite{FeTrZi20} an example of a group, which has neither TBFT$_f$, nor TBFT, was found.
Also, TBFT$_f$ was proved for a class of wreath products \cite{TroLamp,TroitskyWreath2022}.

We prove in Theorem B the TBFT$_f$ for a wide class which is considered by many authors
as a natural extension of the class of polycyclic-by-finite groups.

The interest in Reidemeister theory, $R_\infty$, and TBFT comes from topological
dynamics and other fields (see an overview in \cite{FelshB,polyc}). For some overview of results 
related to $R_\infty$ and recent results one can see 
\cite{FelTroJGT,gowon1,Nasybullov2016JA} (these are several papers from a long publications list).  

Finite rank groups (nilpotent torsion-free) in the context of Reidemeister theory were studied 
in \cite{FelKlop2022,FelTroNilpFR2021RJMP}. Here we study a much wider class.

\smallskip
The paper is organized as follows. In Section \ref{sec:prelim} we recall and prove some necessary statements. Theorem A is proved in Section \ref{sec:jabaratype}.
In Section \ref{sec:abel} we prove statements for abelian quotients to be used in the next section.
Theorem B is proved in Section \ref{sec:tbftforfiniterank}.

\section{Preliminaries and reminding}\label{sec:prelim}

Suppose, that  $H$  is a normal subgroup of $G$ and $p:G\to G/H$ is the natural projection.
Let $H$ be invariant for an endomorphism $\f:G\to G$.
Denote by $\f':H\to H$ and by $\overline{\f}:G/H \to G/H$ the induced homomorphisms.
Evidently, $p$ induces a surjection of Reidemeister classes and
\begin{equation}\label{eq:epi}
R(\overline{\f})\le R(\f).
\end{equation}

\begin{dfn}
	\rm Denote by $C(\f)$ the subgroup of $G$, formed by elements fixed by $\f:G\to G$: $C(\f):=\{g\in G \colon \f(g)=g\}$.
\end{dfn}

The following theorem appeared in \cite{go:nil1}, see also \cite{polyc,GoWon09Crelle}.

\begin{teo}\label{teo:ext}
	For the above defined $G$, $H$, $\f$, $\f'$, and $\widetilde{\f}$,  we have the following properties.
	\begin{itemize}
		\item[a)] If $|C(\overline{\f})|=n$, then $R(\f')\le R(\f)\cdot n$;
		\item[b)] If $C(\overline{\f})=\{e\}$, then each Reidemeister class of $\f'$ is an intersection of the appropriate Reidemeister class of $\f$ and $H$; 
	\end{itemize}	
\end{teo}

Considering the equality
$$
yg\f(y^{-1})x= ygx x^{-1}\f(y^{-1})x= y(gx) (\tau_{x^{-1}} \circ\f)(y^{-1}),
$$
where $\tau_x (z):=  x z x^{-1}$,
one can deduce the following very useful statement (see e.g. \cite{polyc}).
\begin{lem}\label{lem:shift}
The shifts of Reidemeister classes of $\f$ are Reidemeister classes of $\tau_{x^{-1}} \circ \f$:
$$
\{g\}_\f x= \{gx\}_{\tau_{x^{-1}} \circ \f}.
$$	
In particular, $R(\tau_g\circ \f)=R(\f)$.
\end{lem}

\begin{teo}[\cite{Jaikin2002}]\label{teo:Jaikin2002}
Let $G$ be a finite group of rank $r$ admitting an automorphism
with $m$ fixed elements. Then $G$ has a characteristic soluble subgroup $H$, whose
index is $(m, r)$-bounded and whose derived length is $r$-bounded.
\end{teo}

\begin{dfn}\rm
Denote the bounding functions for index and derived length from Theorem \ref{teo:Jaikin2002} by $I(m,r)$ and $L(r)$ respectively.
\end{dfn}

The following statement is well known (see \cite[Lemma 3.5]{MannSegal1990}).
We sketch here a more simple proof with a more weak estimation. 
\begin{teo}\label{teo:num_sub_frank}
The number $a_n(G)$ of subgroups of index $n$ in a finite upper rank group $G$ is finite.
More specifically $a_n(G) \le (n!)^r$, where $r=\urk(G)$. 
\end{teo}

\begin{proof}
Suppose that $G$ has $s$ distinct subgroups of index $n$: $H_1,\dots,H_s$ (we do not suppose that the list is exhaustive). Consider $H=\cap_i H_i$ and the natural
left action of $G$ on the finite set of all left cosets by $H$. Let $H_0$ be the kernel of this action, which is a normal subgroup. Since $H_0 \subseteq H$, the natural projection $p: G \to G/H_0 =:F$ sends 	
$H_i$ to distinct subgroups of index $n$ of the finite group $F$. Thus, $s \le a_n(F)$.
For a finite (or more generally, finitely generated) group $F$ the following estimate is well known
from M.~Hall theorem
(see e.g. \cite[Vol.~1, p.~57]{kurosh}): $a_n(F)\le (n!)^d$, 
where $d$ is the number of elements in some generating system of $F$.
Since $d\le \urk(F)\le \urk(G)$, we are done.
\end{proof}	

Of course, this is a very general rough estimation. For more sharp ones in specific cases see e.g. \cite{LubotSegalBook}.

In \cite{Jabara} E.~Jabara has proved 
\begin{lem}\label{lem:jabara_fin}
Suppose that $\f:F\to F$ is an automorphism of a finite group $G$.
Then $|C(\f)|\le \alpha(R(\f))$, where $\alpha(r)=2^{2^r}$.
\end{lem}

Using this estimation, it was proved in \cite{Jabara} that the same holds for an automorphism of finite order
of a residually finite group, and in \cite{FelLuchTro} for an automorphism of a finitely generated residually finite group.
In the same way we prove here the following statement.
\begin{lem}\label{lem:jabara_fr}
Suppose that $\varphi: G\to G$ is an automorphism of a residually finite group of finite upper rank and
$R:=R(\varphi)<\infty$. Then $|C(\f)|\le \alpha(R)<\infty$. 
\end{lem}

\begin{proof}
Let $\{f_1,\dots,f_n\}\subset C(\f)$ be some finite set of $\f$-fixed elements.
Let $H\subset G$ be a normal subgroup o finite index such that, for $p:G\to G/H$, the elements $p(f_i)$ are pairwise distinct.
By Theorem \ref{teo:num_sub_frank}, we can take the intersection of all subgroups of the same index as $H$
to obtain a characteristic (in particular, normal $\f$-invariant) subgroup of finite index $H'\subseteq G$.
For $p':G\to G/H'$, we still have that  $p'(f_i)$ are pairwise distinct.
Also, $\{p'(f_1),\dots,p'(f_n)\}\subset C(\overline{\f})$, where $\overline{\f}: G/H'\to G/H'$ 
is the induced automorphism of a finite group. Then, by Lemma \ref{lem:jabara_fin}
we have
$$
n\le |C(\overline{\f})| \le \alpha(R(\overline{\f})) \le \alpha(R).
$$
Hence, $|C(\f)|\le \alpha(R)$.
\end{proof}

We will work with the following reformulation of the TBFT$_f$. With some additional restriction it was proved 
in \cite{FelTroJGT}.

\begin{lem}\label{lem:equiv_tbftf}
Suppose that $\f$ is an automorphism of $G$ with $R(\f)<\infty$. Then TBFT$_f$ is true for $\f$ if and only if there is a finite group $F$ with an automorphism $\f_F$ and an equivariant epimorphism
$G \to F$ inducing a bijection of Reidemeister classes of $\f$ and $\f_F$.
\end{lem}

\begin{proof}
Suppose that such an epimorphism $p$ exists. Denote $R=R(\f)$.
Then there are $\widehat{\f_F}$-fixed irreducible unitary pairwise non-equivalent representations $\rho_1,\dots,\rho_R$ of $F$. 
Hence,  $\rho_1 \circ p,\dots,\rho_R\circ p$ are $\widehat{\f}$-fixed irreducible unitary pairwise non-equivalent representations of $G$. Thus $R$ does not exceed 
the number of $\widehat{\f}$-fixed representation classes. 

Each finite-dimensional $\widehat{\f}$-fixed class of representation $\rho$ has a matrix coefficient
$g\mapsto   \operatorname{Trace} (J\circ \rho(g))$, where $J$ is an intertwining operator
between $\rho$ and $\rho\circ \f$. This coefficient is a $\f$-class function (i.e. it is constant
on each Reidemeister class) and only coefficients of this form have this property  (see \cite{polyc} for details of verification). On the other hand, matrix elements of non-equivalent 
finite-dimensional representations are linear independent 
(see \cite[Corollary (27.13)]{CurtisReiner}).
Thus the number of fixed representation classes can not exceed $R$ and the `if' direction is proved.

Conversely, suppose that TBFT$_f$ is true and $\rho_1,\dots,\rho_R$ are $\widehat{\f}$-fixed
irreducible unitary pairwise non-equivalent representations of $G$. Then these representations 
are finite, i.e. $|\rho_i(G)|<\infty$, $i=1,\dots,R$, because otherwise the corresponding
$\f$-class function has infinitely many values. 
Indeed, for $\rho=\rho_1,\dots,\rho_R$, $f_J(g)=\operatorname{Trace} (J\circ \rho(g))$
has to take finitely many values because $R<\infty$.
Since $f_J$ is a non-trivial matrix coefficient,
 its left translations $L_g f_J$ generate a finite-dimensional representation $\pi$, which is equivalent
to a direct sum of several copies of $\rho$. This follows from the fact, that each matrix coefficient is of the form $\f_A(g)=\operatorname{Trace} (A\circ \rho(g))$ and one obtains an equivariant isomorphism $\f_A\mapsto A$, $A\in \End(W_\rho)=W_\rho \otimes W^*_\rho$
between the space of all matrix coefficients and $W_\rho \otimes W^*_\rho$, 
where $W_\rho$ is the space of the representation $\rho$ and the action on $W^*_\rho$ is trivial. 
(This can be extracted e.g. from the proof of \cite[Theorem 3, p.~124]{Kirillov}.)

The space $S$ of the representation $\pi$ has a basis $L_{g_1} f_J,\dots,L_{g_k} f_J$,
for some finite collection $\{g_1,\dots,g_k\}$.
Then the level sets of any function from $S$ have the form $\cap_i g_i U_j$, where
$U_j$ are the level sets of $f_J$. In particular, it has finitely many values.
Taking some unions of these sets (if necessary) we can form a finite partition
$G=V_1\sqcup\dots\sqcup V_m$ such that elements of $S$ are
constant on $V_i$, $j=1,\dots,m$, and separate them.
Thus any left translation permutes $V_i$  and the representation $G$ on $S$ factorizes
through (a subgroup of) the permutation group on $m$
elements, i.e. a finite group. Then the subrepresentation $\rho$ is finite as well.
It remains to take $H=\cap_{i=1}^R \Ker(\rho_i)$. Then $p: G\to G/H$ is the desired equivariant epimorphism onto a finite group, where on $G/H$ we consider the induced automorphism 
$\overline{\f}$. Indeed, $G/H \subseteq \rho_1(G) \oplus\cdots \oplus \rho_R(G)$.
Hence it is finite. By construction, the map is equivariant. It remains to prove that $R(\overline{\f})=R$. For this purpose we will find $R$ distinct $\widehat{\overline{\f}}$-fixed classes of
representations. 
These classes are $[\rho_i \circ p]$. Indeed, since $\rho_i$ are pairwise non-equivalent and 
$p$ is an epimorphism, then so are  $\rho_i \circ p$.
Since 
$$
\widehat{\f}[\rho_i \circ p]=[\rho_i \circ p \circ \f]=[\rho_i  \circ \overline \f \circ p]
=[\rho_i  \circ p],
$$  
the classes are $\widehat{\overline{\f}}$-fixed and we are done.
\end{proof}

We will need the following statement.

\begin{lem}\label{lem:fin_rank_fin_inn}
	Suppose that $G$ is a finite-by-abelian group, where  $F$ is a finite normal subgroup and the abelian factor group
	$A=G/F$ is of finite Pr\"ufer rank $\rk(A)<\infty$.
	Then the number of  distinct inner automorphisms of $G$ does not exceed $|F|!\cdot |F|^{\rk(A)}$.
\end{lem}

\begin{proof} 
Let $A_0\subseteq A$ be a subgroup with generators $a_1,\dots ,a_k$, $k\le \rk(A)$. Then $G_0=p^{-1}(A_0)$ has a generating set 	$F\cup\{g_1,\dots,g_k\}$ of cardinality $|F| + k$, where 
$p$ is the projection $p:G\to A$ and
$g_i\in G_0$ are arbitrary elements such that 	$p(g_i)=a_i$, $i=1,\dots,k$.
Then any inner automorphism $\t_g$ of $G_0$ is completely defined by the images of these generators. Since $A$ is abelian, $\t_g$ maps each $F$-coset to itself. Hence, we obtain not more than $|F|!\cdot |F|^k$ possibilities.
	
Now suppose that for the entire $G$ there is $s> |F|!\cdot |F|^{\rk(A)}$ distinct inner automorphisms
	$\tau_{y_i}$, $y_i\in G$, $i=1,\dots,s$. 
Then there are elements $x_{ij}\in G$ such that $\tau_{y_i}(x_{ij})\ne \tau_{y_j}(x_{ij})$.
	Consider $A_0\subseteq A$ generated by all $p(y_i)$ and $p(x_{ij})$ and take $G_0:=p^{-1}A_0$.
	Then $\tau_{y_i}$ remain inner automorphisms of $G_0$ (because $y_i\in G_0$) and remain distinct. 
This contradicts the first part of the proof.
\end{proof}

\begin{cor}\label{cor:tbft_for_finite_by_ab}
Suppose, the number of distinct inner automorphisms of $G$ is finite.
Then the TBFT$_f$ is true for $G$.
In particular, under suppositions of Lemma \ref{lem:fin_rank_fin_inn} the TBFT$_f$ is true for $G$.
\end{cor}

\begin{proof} 
Let $\f:G\to G$ be an automorphism with a finite Reidemeister number.
Consider shifts of Reidemeister classes. By Lemma \ref{lem:shift}, there is only finitely many such sets,
namely $n \le R(\f)\cdot t$, where $t$ is the number of distinct inner automorphisms.
Thus there is a natural homomorphism $\sigma: G \to S_n$, where $S_n$ is the permutation group.
Denote by $R_n$ its image. The kernel of $\sigma$ is the intersection of shift stabilizers of Reidemeister
classes of $\f$.  Each of these stabilizers is $\f$-invariant, because, if $z$ belongs to the stabilizer 
$\St(\{x\}_\f)$
of $\{x\}_\f$ then
\begin{eqnarray*}
gx\f(g^{-1})\f(z)&=&gx\f(x)\f(x^{-1} g^{-1}) \f(z)=\f\left[ \f^{-1} (gx) x \f(\f^{-1}(x^{-1}g^{-1}))z\right]\\
&=&\f(hx\f(h^{-1}))=(\f(h)x^{-1}) x \f(x \f(h^{-1})),
\end{eqnarray*}
for some $h$.
So $\Ker\sigma$ is $\f$-invariant and we can consider the induced automorphism $\f_R:R_n\to R_n$.
We claim that $\sigma:G\to R_n$ induces a bijection of Reidemeister classes. To prove this, we need to verify
that any Reidemeister class of $\f$ is a union of cosets of $\Ker\sigma$. 
Consider an arbitrary element $x\in \{x\}_\f$ and  $z \in \St(\{x\}_\f)$. Then
$
x z = g x \f(g^{-1}),
$  
for some $g$. Hence, $x \in x  \Ker\sigma \subseteq x \St(\{x\}_\f) \subseteq \{x\}_\f$.
Thus, each element is included in its Reidemeister class together with some coset   of $\Ker\sigma$ and we are done.
\end{proof}

\section{Proof of Theorem A}\label{sec:jabaratype}

\begin{proof}[Proof of Theorem A]
Our argument will be close to some parts of \cite{Jabara}. 	
For a residually finite group $G$ consider an exhausting family of subgroups
of finite index $H_\lambda$, $\lambda\in \Lambda$. By Theorem \ref{teo:num_sub_frank},
taking the intersection of all subgroups of a fixed index, we may assume $H_\lambda$ to be characteristic, in particular normal and $\f$-invariant. Denote $F_\lambda:=G/H_\lambda$ and
by $\f_\lambda: F_\lambda\to F_\lambda$ the induced automorphism. Then $R(\f_\lambda)\le R(\f)=:R<\infty$.
Hence by Lemma \ref{lem:jabara_fr} we have $|C(\f_\lambda)|\le \alpha(R)$. Of course, 
$\rk(F_\lambda)=\urk(F_\lambda)\le \urk(G)=r$. 
Applying Theorem \ref{teo:Jaikin2002} we can find, for each $\lambda\in\Lambda$, 
a soluble characteristic subgroup  	
$F^0_\lambda\subseteq F_\lambda$ of length $\le L(r)=:L$ and index $\le I(\alpha(R),r)=:J$.	
For each $\lambda$, 
denote by $p_\lambda:G\to F_\lambda$ the natural projection and 
by $H^0_\lambda$ the pre-image $H^0_\lambda=(p_\lambda)^{-1}(F^0_\lambda)$. 
Let $H:=\cap_\lambda H^0_\lambda$. Then each $H^0_\lambda$ is a characteristic subgroup of index $\le J$.
Then $H$ is characteristic and of finite index by Theorem \ref{teo:num_sub_frank}.

It remains to prove that $H$ is soluble. Suppose the opposite: $H$ is not a soluble group of derived length $\le L$. Thus there exists an expression $\omega(h_1,\dots,h_k)\ne e$ in products and commutators, which has the commutator length $L+1$. Choose $\lambda$ such that 
\begin{equation}\label{eq:nontriv_comm}
	p_\lambda (\omega(h_1,\dots,h_k))\ne e\in F_\lambda.
\end{equation}
Since $p_\lambda (\omega(h_1,\dots,h_k))=\omega(p_\lambda (h_1),\dots,p_\lambda (h_k))$, 
$h_i\in H \subseteq H_\lambda$ and $p_\lambda (h_i)\in F^0_\lambda$, the inequality (\ref{eq:nontriv_comm})
contradicts the property of $F^0_\lambda$ to be soluble of 	length $\le L$. Thus $H$ is a characteristic
soluble subgroup of finite index in $G$.
\end{proof}

\begin{rmk}
	\rm
Any \emph{finitely generated} residually finite group of finite upper rank 
is soluble-by-finite by Mann-Segal \cite{MannSegal1990}.
\end{rmk}

\section{Abelian finite rank groups and their automorphisms}\label{sec:abel}
We need a separate (from \ref{lem:jabara_fr}) study of abelian quotients, which may be not residually finite.
We start from the torsion-free case.

\begin{lem}\label{lem:no_fixed_for_free} {\rm (cf.~\cite{FelTroNilpFR2021RJMP})}
	An automorphism $\f$ of any torsion free abelian group $A$ of finite rank with finite Reidemeister number has $C(\f)=\{e\}$.
\end{lem}

\begin{proof}
By Fuchs' lemma (see \cite[15.2.3]{Robinson}),
an endomorphism  $\psi$ of a torsion-free abelian group $A$ of
	finite rank is injective if and only if the index of $\Im(\psi)$ is finite.
Since $R(\f)=|\Coker(1-\f)|<\infty$, we can apply this lemma to $\psi=1-\f$
and obtain
$\{e\}=\Ker(1-\f)=C(\f)$.	
\end{proof}	

Now suppose that $\f:A\to A$ is an automorphism of a general finite rank abelian group with $R(\f)<\infty$.
Let $T$ be the torsion subgroup, hence $A_0:=A/T$ is torsion-free. The subgroup $T$ is fully invariant.
Denote by $\f_T:T\to T$ and $\f_0:A_0\to A_0$ the automorphisms induced by $\f$.

By Lemma \ref{lem:no_fixed_for_free}, $C(\f_0)=\{e\}$. 
Hence, by Theorem \ref{teo:ext}, $R(\f_T)<\infty$.
Also,
\begin{equation}\label{eq:fixed_on_T}
	C(\f)=C(\f_T).
\end{equation}

For $p$ prime, denote by $\Z(p^\infty)$ the \emph{quasicyclic} (or \emph{Pr\"ufer}) \emph{group of type $p^\infty$}.
An abelian finite rank torsion group is a direct sum of finitely many quasicyclic and cyclic groups (see e.g.
\cite[4.1.1 + 4.3.13]{Robinson}):
\begin{equation}\label{eq:decomp_of_T}
	T=T_1\oplus \cdots \oplus T_m \oplus F, \qquad T_i:= \Z(p_i^\infty)^{d_i},
	\quad p_i\ne p_j\mbox{ if }i\ne j, \qquad |F|<\infty.
\end{equation}

We will need the following simple observation.
\begin{lem}\label{lem:quasi_sub}
Suppose that $G=\Z(p^\infty)^d$ (finite direct sum) and $H$ is a subgroup of $G$.
Then either $H$ is finite or $G/H\cong \Z(p^\infty)^{d'}$ with $d'<d$.  	
\end{lem} 

\begin{proof}
The subgroup $H$ is a finite rank
$p$-group, hence $H=H'\oplus F$, where $H'\cong \Z(p^\infty)^{m'}$ and $F$ is a finite 
sum of cyclic groups. The subgroup $H$ is infinite if and only if $m'\ge 1$.
Since $H'$ is a divisible subgroup, it has a complement $H''$, so $G=H'\oplus H''$.
Then $H''\cong G/H'$ is a divisible $p$-group. Hence, $H''\cong \Z(p^\infty)^{m''}$. 
Similarly, $G/H\cong \Z(p^\infty)^{d'}$.
Since the number of $\Z(p^\infty)$-summands is an invariant,
$d=m'+m''$. Since $H''$ has a subgroup $F'$, isomorphic $F$, and $H''/F'\cong G/H$, then $d'\le m''=d-m'$.
Hence, if $H$ is infinite, then $d'<d$.
\end{proof}	

 \begin{lem}\label{lem:quasi_fix}
Suppose that $G=\Z(p^\infty)^d$ and $\f:G\to G$ is an endomorphism with $R(\f)<\infty$.
Then $|C(\f)|<\infty$. 	
 \end{lem} 

\begin{proof}
Consider $C(\f)=\Ker(1-\f)$, which is a $\f$-invariant subgroup.
If $|C(\f)|=\infty$, consider $\f_1: G_1\to G_1$, where $G_1=G/C(\f)$. If $|C(\f_1)|=\infty$ then continue the process.
By Lemma \ref{lem:quasi_sub} the process will stop in several steps.  
At that moment the following cases can occur: 

1) $C(\f_i)=G_i$ and the next factor group will be trivial. Then $\f_i$ is the identity map of an infinite abelian group. Hence,
by Theorem \ref{teo:ext} $\infty=R(\f_i) \le R(\f)<\infty$. A contradiction.

2) $|C(\f_i)|<\infty$, but $|C(\f_{i-1})|=\infty$ (otherwise we stop at the previous step). (Here we denote $G_0=G$ and $\f_0=\f$.)
By Theorem \ref{teo:ext} $R(\f_{i-1})<\infty$.
Then by Theorem \ref{teo:ext} $R(\f'_{i-1})<\infty$, where $\f'_{i-1}=\f_{i-1}|_{C(\f_{i-1})}$. But $\f'_{i-1}$ is the identity map of an infinite abelian group
and $R(\f'_{i-1})=\infty$. That is why there were no previous steps, i.e. $i=0$, and $|C(\f)|<\infty$ as desired.
\end{proof}	

Since $T_1\oplus \cdots \oplus T_m$ is fully invariant as the divisible part,
and each its $p_i$-component $T_i$ is also invariant, we have
$$
C(\f_T)\le |F|\cdot \prod_{i=1}^{m} C(\f_{T_i}),\mbox{ where }\f_{T_i}=\f_T|_{T_i}.
$$  
Also, by Theorem  \ref{teo:ext},
$$
\prod_{i=1}^{m} R(\f_{T_i})=R(\f_T|T_1\oplus \cdots \oplus T_m) \le R(\f_T) \cdot |F| <\infty.
$$
Applying Lemma \ref{lem:quasi_fix} we obtain the main statement of this section:
\begin{teo}\label{teo:fin_fix_Ai}
Let $\f$ be an automorphism of an abelian group of finite rank with $R(\f)<\infty$. Then $|C(\f)|<\infty$.
\end{teo}

\section{Proof of Theorem B}\label{sec:tbftforfiniterank}

\begin{proof}[Proof of Theorem B]
Suppose, $\f:G\to G$ is an automorphism with $R(\f)<\infty$. Then Theorem A implies that $G$ is soluble-by-finite.

Let us remark that the soluble subgroup $H$ of finite index constructed in 
Theorem A is characteristic and one can reduce 
the proof of TBFT$_f$ to the case of a soluble $G$. Indeed, for the projection $p:G\to G/H$ and the
induced automorphism $\overline{\f}:G/H\to G/H$, consider some representatives 
$g_1 ,\dots, g_r$ (with $g_1=e$) of
Reidemeister classes of  ${\f}$. Denote by $\f'$ the restriction of $\f$ to $H$.
For each $\t_{g_j}$, $j=1,\dots,r$, one has $R(\t_{g_j} \circ \f')<\infty$ by Theorem \ref{teo:ext} and since
$|G/H|<\infty$.
If $H$ has the TBFT$_f$ property, one can find, for each $j$,  a normal $\t_{g_j} \circ  \f'$-invariant subgroup
$H_j$ such that $p_j: H\to H/H_j$ gives a bijection of Reidemeister classes of  $\t_{g_j} \circ  \f'$ and of the
induced automorphism of $H/H_j$.
Since the rank of $H$ is finite,  one can consider the intersection $H'_j$ of all subgroups of $H$ of the same index as
$H_j$. Then by \ref{teo:num_sub_frank} $H'_j$ is a characteristic subgroup of $H$ of finite index.
The same is true for $H^0:=\cap_{j=1}^r H'_j$. So, we can consider an equivariant epimorphism
$p^0: G\to G/H^0$, where $G/H^0$ is finite. It remains to verify that $p^0$ gives a bijective mapping of 
Reidemeister classes. As in the proof of Corollary \ref{cor:tbft_for_finite_by_ab}, we need to prove that Reidemeister classes of $\f$ are some unions of cosets of $H^0$.
For this purpose, consider an arbitrary $g\in G$. Then, for some $j$, $j=1,\dots,r$,
we have $g\in \{g_j\}_\f$. Hence, $g=x g_j \f(x^{-1})$, for some $x\in G$.
Then
$
H^0 g = H^0 x g_j \f(x^{-1}) = x H^0 g_j \f(x^{-1}) 
$.
Since $H^0\subseteq \{e\}_ {\t_{g_j} \circ  \f'}$, for any $h_0 \in H^0$ and some $h\in H$,
we have $h_0=h (\t_{g_j} \circ  \f')(h^{-1})$. Then
$$
x h^0 g_j \f(x^{-1}) = x h g_j  (\f')(h^{-1}) (g_j)^{-1} g_j \f(x^{-1}) =
x h g_j  (\f')(h^{-1})  \f(x^{-1}) 
 \in \{g_j\}_\f.
$$
Hence,
\begin{equation}\label{eq:inner_and_quot}
g\in H^0 g \subseteq \{g_j\}_\f. 
\end{equation}
So any class is a union of $H^0$-cosets and the reduction is proved. 

Thus we suppose $G$ to be soluble and consider the derived series of $G$:
\begin{equation}\label{eq:deriv_G}
	G=G^{(0)}\supset G'=G^{(1)} \supset G^{(2)} \supset \cdots \supset G^{(k)}=\{e\},
	\qquad G^{(i+1)}=[G^{(i)},G^{(i)}], 
\end{equation}
of characteristic subgroups. Denote by $A_i:=G^{(i)}/G^{(i+1)}$ the corresponding
abelian groups and by $G_i:=G/G^{(i)}$ the corresponding factor groups. 
Denote the induced automorphism by
$$
\f_i:=G^{(i)}\to G^{(i)},\qquad \a_i: A_i\to A_i, \qquad \b_i: G_i \to G_i.
$$

First we will prove inductively that $R(\f_i)<\infty$, and similarly, by Lemma \ref{lem:shift}
applied to the initial $\f$,
that
\begin{equation}\label{eq:R_is_fin_for_i}
R(\t_g\circ \f_i)<\infty, \qquad\mbox{ for any }g\in G,\quad i=0,\dots,k.
\end{equation}
Indeed, by Theorem \ref{teo:fin_fix_Ai},  $|C(\a_0)|<\infty$. Then, by Theorem \ref{teo:ext}.a),
$R(\f_1)<\infty$. Since the finite rank condition is hereditary for subgroups and factor groups,
the next steps of the induction are quite similar to the first one.  

Now we will argue by induction to prove the theorem.
Since $G^{(k-1)}$ is abelian, TBFT$_f$ is true for it. Suppose that the theorem is proved for $G^{(i+1)}$ and prove it for $G^{(i)}$.

Choose representatives $e=g_1,\dots,g_r$ in $G^{(i)}$ of Reidemeister classes of $\f_i$ 
For each $g_j$ ($j=1,\dots,r$), choose a subgroup $H_j$ realizing the TBFT$_f$ for $\t_{g_j}\circ \f_{i+1}$.
In other words, $H_j$ is a normal $\t_{g_j}\circ \f_{i+1}$-invariant subgroup of $G^{(i+1)}$ and the projection
$G^{(i+1)}\to G^{(i+1)}/H_j$ induces a bijection of Reidemeister classes of $\t_{g_j}\circ \f_{i+1}$ and of those of the induced automorphism of $ G^{(i+1)}/H_j$ (equivalently each Reidemeister class of $\t_{g_j}\circ \f_{i+1}$ is a union of some cosets of $H_j$). 

Consider $H_0:=\cap_{j=1}^r H_r$. This is a subgroup of finite index in $G^{(i+1)}$.
By Theorem \ref{teo:num_sub_frank}
the intersection $H$ of all subgroups of $G^{(i+1)}$ of the same index as $H_0$ is a characteristic subgroup of $G^{(i+1)}$ of finite index. We have $H\subseteq H_0 \subseteq H_j$ ($j=1,\dots,r$). 

Then $G^{(i)} \to G^{(i)}/H$ induces a bijection of Reidemeister classes. This can be proved literally like for $G\to G/H^0$ in the beginning of this proof.
Finally $G^{(i)}/H$ is finite-by-abelian, namely $(G^{(i+1)}/H)$-by-$A_i$. Since $A_i$ has finite rank, we apply Lemma \ref{lem:fin_rank_fin_inn} and Corollary \ref{cor:tbft_for_finite_by_ab}
to see that TBFT$_f$ is true for $G^{(i)}/H$ and there is an equivariant projection onto a finite group $G^{(i)}/H\to F$ inducing a bijection of Reidemeister classes. Then the composition $G^{(i)}\to G^{(i)}/H\to F$ gives a bijection of Reidemeister classes,
which proves the TBFT$_f$  for $G^{(i)}$. The induction completes the proof.
\end{proof}


\end{document}